%% file: SCL_Main.tex
\documentclass[journal,twoside]{IEEEtran}

\usepackage{epsfig}

\usepackage{amsmath,amssymb,amsthm,url,amsfonts,mathrsfs, bbm}
\usepackage{graphicx}
\usepackage{float}
\usepackage{cite}

\usepackage{color}

\def\old#1{}

\newtheorem{assumption}{Assumption}

\newtheorem{example}{Example}
\newtheorem{remark}{Remark}
\usepackage{enumerate}
\usepackage{verbatim}

\newcommand{\be}{\begin{equation}}
\newcommand{\ee}{\end{equation}}

\renewcommand{\v}[1]{\ensuremath{\boldsymbol{\mathrm{#1}}}}

\newtheorem{theorem}{Theorem}[section]
\newtheorem{lemma}[theorem]{Lemma}

\newtheorem{corollary}[theorem]{Corollary}
\newtheorem{proposition}[theorem]{Proposition}
\newtheorem{definition}[theorem]{Definition}

\newcommand{\E}{\mathbb{E}}

\providecommand{\boldsymbol}[1]{\mbox{\boldmath $#1$}}

\makeatother

\input{header.tex}

\begin{document}

\title{On the Marginal Value of Electricity Storage$^*$}

\author{Eilyan Bitar$^a$,   Pramod Khargonekar$^b$,  \ and \ Kameshwar Poolla$^c$   
\thanks{
$^*$This work was supported in part by NSF grants ECCS-1351621, CNS-1239178,  CNS-1239274,  and  IIP-1632124.
This work builds on our preliminary
results, presented at the 30th IEEE American Control Conference \cite{BitarACC2011}. The current manuscript differs significantly from the conference version in terms of new results, formal proofs, and
detailed technical discussions.}
\thanks{$^a$Corresponding author: E. Bitar  is with the School of Electrical and Computer Engineering, Cornell University, Ithaca, NY 14853 USA. ({\tt\small eyb5@cornell.edu})}
\thanks{$^b$P. Khargonekar is with the University of California, Irvine, CA 92697 USA. ({\tt\small pramod.khargonekar@uci.edu })}%
\thanks{$^c$K. Poolla is with the Department of Electrical Engineering and Computer
Science, University of California, Berkeley, CA 94704 USA. ({\tt\small poolla@berkeley.edu})}
}

\date{\today}
\maketitle

\begin{abstract}
We investigate the problem of characterizing the economic value of  energy storage capacity to a wind power producer (WPP) that sells its energy  in a conventional two-settlement electricity market. The WPP can offer a forward contract to supply power in the day-ahead market, subject to financial penalties for imbalances between the contracted power and the power that is delivered in real-time. We consider the setting in which the WPP has access to a co-located energy storage system, and can thus reshape its wind power production subject to its storage capacity constraints.  Modeling wind power as a random process, we  show that the problem of determining optimal forward contract offerings---given recourse with storage---is convex. We further establish that the maximum expected profit is concave and non-decreasing in the energy storage capacity, which reveals that the greatest marginal benefit from energy storage is derived from initial investment in small storage capacity. We provide a characterization of the marginal value of small energy storage capacity to the WPP. The formulae we derive shed light on the relationship between the value of storage and certain statistical measures of variability in the underlying wind power process. 

\end{abstract}

\input{introduction}

\input{formulation}

\input{results_structure}

\input{results_marginal_val}
\input{conclusion}

\input{appendix}

\bibliographystyle{IEEEtran}
\bibliography{references_bib}{\markboth{References}{References}}           

\end{document}

%% file: header.tex
\newcommand{\beq}{\begin{equation}}
\newcommand{\eeq}{\end{equation}}
\newcommand{\beqa}{\begin{eqnarray}}
\newcommand{\eeqa}{\end{eqnarray}}
\newcommand{\beqan}{\begin{eqnarray*}}
\newcommand{\eeqan}{\end{eqnarray*}}

\newcommand{\bmat}[1]{ \begin{bmatrix} #1 \end{bmatrix}}

\newcommand{\Nset}{\mathbb{N}}

\newcommand{\Rset}{\mathbb{R}}

\newcommand{\Ical}{{\cal I}}

\newcommand{\Bcal}{{\cal B}}

\newcommand{\Dcal}{{\cal D}}

\newcommand{\Kcal}{{\cal K}}

\newcommand{\Scal}{{\cal S}}

\newcommand{\Ucal}{{\cal U}}
\newcommand{\Xcal}{{\cal X}}
\newcommand{\Zcal}{{\cal Z}}

\newcommand{\blue}[1]{{#1}}

\newcounter{l1}
\newcounter{l2}
\newcounter{l3}
\newcommand{\bdotlist}{\begin{list}{$\bullet$}{}}
\newcommand{\bboxlist}{\begin{list}{$\Box$}{}}
\newcommand{\bbboxlist}{\begin{list}{\raisebox{.005in}{{\tiny
$\blacksquare$ \ \ }}}{}}
\newcommand{\bdashlist}{\begin{list}{$-$}{} }
\newcommand{\blist}{\begin{list}{}{} }
\newcommand{\barablist}{\begin{list}{\arabic{l1}}{\usecounter{l1}}}
\newcommand{\balphlist}{\begin{list}{(\alph{l2})}{\usecounter{l2}}}
\newcommand{\bAlphlist}{\begin{list}{\Alph{l2}.}{\usecounter{l2}}}
\newcommand{\bdiamlist}{\begin{list}{$\diamond$}{}}
\newcommand{\bromalist}{\begin{list}{(\roman{l3})}{\usecounter{l3}}}

\newcommand{\thm}[1]{\noindent \begin{theorem} #1   \end{theorem}}



%% file: introduction.tex
\section{Introduction}
Driven by concerns of climate change and energy security, there is a growing worldwide investment in renewable energy \cite{GWEC2010}. The available supply of power from sources like wind and solar is {\em variable}---it is uncertain, intermittent, and largely uncontrollable. These characteristics pose major challenges to the deep  integration of renewables into the  grid \cite{EWITS,IVGTF2009,WWTS}.

There is considerable investment and interest in energy storage as a means to mitigate the variability of renewable generation
\cite{castillo2014grid, cavallo2, decesaro, denholmnrel, divya, korpaas03}. 
For instance, California's strategic storage mandate calls for 1.3 GW of ramping capability to be 
commissioned by 2020.
Hydro-power has traditionally been used for such purposes \cite{cast}. 
While pumped hydro is an efficient and flexible storage modality, its siting is geographically constrained and thus offers 
limited balancing capability to the system at large due to transmission constraints. As utility-scale renewable energy resources continue to proliferate, the ability to directly shape their power output with alternative forms of electric energy storage becomes more compelling.

In this paper, we investigate the ability of energy storage to mitigate the cost of balancing variable renewable power. The perspective taken is that of a wind power producer (WPP), whose objective is to sell its \emph{variable power} in  conventional, two-settlement energy markets akin to those previously studied in \cite{angarita1,Bathurst2002, bitar_trans, Dent2011, Matevosyan2006,morales10, Pag2014, Pinson2007}. In such markets, imbalances arising between the contracted and realized supply of power are subject to financial penalty. 
Accordingly, we quantify the value of \emph{co-located energy storage} in terms of its ability to reduce the expected cost of such contract imbalances, and, thereby,  increase the profitability of wind power in such markets.

\subsection{Contribution and Related Work}

Energy storage devices such as pumped-hydro, compressed air \cite{katz, lund2,succar},  sodium-sulfur batteries \cite{oshima}, and    more general battery-based  technologies \cite{divya} offer the capability to firm variable wind power. The ability to do so depends centrally on the placement, sizing, and control of such energy storage systems.  \blue{For a review of papers that consider the problem of optimally siting energy storage in (deterministic) transmission-constrained power systems, we refer the reader to \cite{bose2012optimal, DenholmSioshansi2009, thrampoulidis2016optimal}. }
There are also a number of papers that explore the
economic viability of these hybrid wind-storage systems in producing baseload generation 
\cite{cavallo1, cavallo2, greenblatt1}. These studies 
conclude that such hybrid systems compete favorably with gas turbine, conventional fossil fuel, and nuclear generation.

This paper considers the problem of leveraging   energy storage systems to improve the profitability of a wind power producer (WPP) participating in a two-settlement energy market. In the setting considered, the WPP offers  a contract for firm power in the day-ahead market, subject to financial penalty for deviations between the contracted and realized supply in the real-time market. The recourse afforded by storage serves to reduce the risk exposure of the WPP, enabling it to offer larger contracts, which in turn increases its expected profit. 
We show that the optimal contract sizing problem reduces to convex programming. 
We also prove that the maximum expected profit of the WPP is a concave and non-decreasing function of the storage capacity. As a result, the greatest marginal benefit is derived for a \emph{small} energy storage capacity. In Theorem \ref{thm:mv}, we provide an explicit characterization of
the marginal value of small energy storage capacity  in terms of a specific statistical measure of  variation in the underlying wind power process.




\blue{
There are a number of  related papers in the literature, which attempt to characterize  the economic value of energy storage capacity across a variety of electricity market settings using either deterministic (offline) optimization methods \cite{de2016value,del2016synergy, lamont2013assessing},  stochastic optimization methods \cite{angarita1, go2016assessing}, online convex optimization methods \cite{Qin2015}, or dynamic programming-based methods \cite{chowdhury2016benefits,  KimPowell2011, harsha2015,  parandehgheibi2015value,  rao2015value,  van2013optimal}. 
 With the exception of \cite{chowdhury2016benefits,KimPowell2011}, the majority of the aforementioned papers calculate the economic value  of storage capacity using sensitivity analyses that are largely numerical  in nature---in contrast to the closed-form marginal value expressions established in this paper.  Closer to the approach adopted in this paper, the authors in \cite{chowdhury2016benefits,KimPowell2011} derive analytical expressions for the marginal value of storage, albeit under the somewhat restrictive assumption that the wind power generated in each time period be uniformly distributed. In contrast, the structural and marginal value results derived in this paper are \emph{distribution-free}, in the sense that they hold for any (possibly nonstationary) wind power process with an absolutely continuous joint distribution. We note, however, that an important limitation of our results is their reliance upon the assumption of constant real-time imbalance prices.}

%
%
%
%
%

\emph{Organization:} The remainder of the paper is organized as follows. Sections \ref{models} and \ref{sec:prob} describe the models that we employ in our analysis, and formulate the specific  problems that we address, respectively.  Our main results are presented in Sections \ref{mainresults} and \ref{sec:marval}, followed by concluding remarks in Section \ref{conclu}.

\emph{Notation:} For any finite set $A$, we denote its cardinality by $|A|$. 
Let $\Nset$ denote the set of non-negative integers, $\Rset$ the set of real numbers, $\Rset_+$ the non-negative reals, 
and $\Rset^N$ the usual Euclidean space. 
For $x \in \Rset$, let $x^+ = \max\{x,0\}$  and $x^- = \min\{x,0\}$.
For any subset $A \subseteq \Rset^N$, we define the indicator function $\v{1}_A: \Rset^N \rightarrow \{0,1\}$ as 
$$ \v{1}_A(x) = \begin{cases}  1, & x \in A \\ 0, & x \notin A. \end{cases} $$

%% file: formulation.tex
\section{Supply, Storage and Market Models} \label{models}

\subsection{Intermittent Supply Model}
Time is slotted and indexed by $k$. The intermittent generation of the WPP in time period $k$ is $\xi_k$. This is normalized to
nameplate so $\xi_k \in [0,1]$.  We model the wind farm output as a  discrete-time  
random process \ $\v{\xi} = (\xi_0, \xi_1, \dots,\xi_{N-1})$. Let  
\beq
\Phi_k(x) = \mathbb{P}\{ \xi_k \leq x\},
\ \ \ k=0,\dots,N-1
\eeq 
denote the cumulative distribution function of $\xi_k$.  Define the {\em time-averaged} cumulative distribution function as
\beq
F(x) = \frac{1}{N} \sum_{k=0}^{N-1} \Phi_k(x).
\eeq
 We assume that the intermittent supply has a \emph{zero cost of production}, as it is derived from wind and solar energy.

\subsection{Energy Storage Model}
\label{storage_model}
We consider a simple energy balance model \cite{korpaas03} for \emph{perfectly efficient} energy storage\footnote{The assumption of perfectly efficient storage is for ease of exposition. All of our results can be generalized to accommodate non-ideal storage systems with leakage and energy conversion inefficiencies. See Remark \ref{rem:nonid} for a discussion on such generalization.}:

\beq
\label{stor_dynam_disc} z_{k+1} = z_k  - u_k, \quad k =0,1,\cdots
\eeq
Here, $z_k \geq 0$  represents the amount of energy in the storage at the beginning of time slot $k$, and $u_k$  denotes the energy that is extracted from or injected into the storage  during time slot $k$. \blue{The sign convention is such 
that $u_k > 0$ corresponds to an energy extraction, and $u_k < 0$ corresponds to an energy injection.}
 Without loss of generality, we assume a zero initial condition, 
$z_0 = 0$. 
We impose the state and input constraints: 
\begin{eqnarray}
0 \leq & z_k & \leq b  \label{con1} \\
-r \leq & u_k & \leq   r \label{con2}
\end{eqnarray}
to capture the \emph{energy capacity} $b$  and the maximum charging/discharging \emph{rate} $r$. 
We refer to the \emph{storage type} as the parameter vector $\theta = (b,r)$.

\subsection{Market Model}
We consider a two-settlement electricity market  consisting of a day-ahead (DA) market and a real-time (RT) imbalance market.
\blue{In the DA market, a generator can submit offers to produce power over the following day according  to a sequence of power contracts  that are typically piecewise constant over hour-long time intervals.  Normally, the DA market will close for offers by 10 AM,  and clears by 1 PM on the day immediately preceding the delivery day.  The contracts cleared in the DA market
are binding and call for delivery in the RT market, where uninstructed deviations between the contracted power and the delivered power are penalized according to imbalance prices determined in the RT market.}\footnote{We note that the market model considered in this paper conforms with the prevailing literature on the integration of wind power through two-settlement electricity markets  
\cite{Baeyens2013,Bathurst2002,bitar_trans,BitarACC2011,Dent2011,Matevosyan2006,Pinson2007}.}

\subsubsection*{Day-Ahead (DA) Market} 
In this paper, we  restrict our analysis to a single DA contract interval, which we discretize into  $N$ time slots reflecting the finer temporal granularity of  RT market operations.\footnote{\blue{We note that it is straightforward to extend the formulation and results presented in this paper to accommodate the more general setting in which a supplier can offer a sequence of  multiple DA contracts (e.g., one for each hour of the day)  that are remunerated according to the corresponding sequence of hourly DA market  prices. We refer the  reader to Section II-C of \cite{BitarACC2011} for the mathematical details of such a formulation.}}  We let $x \in \Rset$ (MWh)  denote the offered  contract, which is taken to be constant across the $N$ time slots defining the contract interval. The supplier is payed according to  the DA market clearing price $p \in \Rset_+$ (\$/MWh) associated with that contract interval. This yields the supplier a revenue of $N \cdot px$ in the DA market.

 \emph{Real-Time (RT) Market.} \quad  As the forward contract $x$ is  offered with significant lead time on delivery, deviations naturally arise between the offered contract and the delivered power. These contract deviations are penalized according to imbalance prices derived from the RT market.
A shortfall in generation during period $k$ is penalized at a price $\alpha \in \Rset_+$ (\$/MWh), while an excess in generation is penalized at a price $\beta \in \Rset_+$ (\$/MWh). Typically, deviations from hour-long forward contracts are measured on a finer temporal granularity corresponding to intervals of length five minutes.  Accordingly, we consider a temporal discretization of the  contract interval into $N$ discrete time periods, where each period's imbalance is measured relative to  the baseline contract $x$. 
\blue{We also note that, in practice, a WPP may possess the ability to physically curtail its  power output in real-time   by pitching its turbine blades to avoid overproduction imbalance penalties. One can reflect the economic impact of  this curtailment capability by setting the overproduction imbalance price $\beta$ equal to zero.}

%

\blue{ \emph{Market Assumptions.} \quad  
We make several common assumptions regarding  the determination of  prices in the two-settlement  energy market under consideration. First, we
assume that the WPP's production capacity is small
relative to the aggregate capacity of other generators participating in the DA energy market. 
Under this assumption,
it is fair to assume that the WPP cannot appreciably affect the determination of prices.
Accordingly, we require the WPP to behave as a price taker in the
DA energy market, and model the DA energy $p$ as \emph{fixed} and \emph{known} at the time of forward contract offering. We refer the reader to Remark \ref{rem:supplyfunc}, which provides an alternative interpretation of the optimal forward contract offering  as a supply function offer in the DA market.}

Second, as the RT imbalance prices $(\alpha, \beta)$ are
not known to the WPP at the time of committing to a forward contract
in the DA market, we model them as random variables whose expected values
at the time of forward contract offering are denoted by
\begin{align*}
m_{\alpha} = \mathbb{E}[\alpha] \ \ \text{and} \ \ m_{\beta}  =  \mathbb{E}[\beta].
\end{align*}
Additionaly, the RT imbalance prices $(\alpha, \beta)$ are assumed to be \emph{independent} of the intermittent supply process $\v{\xi}$. Again, such an assumption is reasonable if the WPP's production capacity is small relative to the market size, as the WPP's realized contract  deviations will have negligible effect on the  determination of prices in the RT market. \blue{Naturally, this assumption may need to be reexamined for markets scenarios in which the aggregate capacity of participating wind power producers is large. We refer the reader to several recent papers \cite{Dent2011,Pag2014}, which treat the possibility of correlation between imbalance prices and wind power in simpler settings without energy storage. }

Finally, we make the following technical assumption.
\begin{assumption} \label{ass:price}
The DA market price  satisfies  $p \leq m_{\alpha}$.
\end{assumption}
\blue{ 
From a technical perspective,  such an assumption ensures concavity of the WPP's expected profit
function \eqref{eq:profit} in the forward contract $x$. More practically, this assumption eliminates the perverse incentive for the  
 WPP to offer larger forward  contracts in the DA market with
the explicit intention of underproducing in the RT market relative to the offered forward contract.
}


\section{Problem Formulation} 

\label{sec:prob}

 Working within this idealized setting, we now formalize the question of how a generator with intermittent supply might optimize a forward contract offering for energy given a subsequent sequence of recourse opportunities to reshape the realized supply profile using a constrained energy storage device. Building on intermediary results characterizing the structure of the optimal value function, the eventual goal is a parametric sensitivity analysis yielding an explicit characterization of the \emph{marginal value of energy storage capacity}.  We begin by characterizing the space of admissible, causal storage control policies.

\subsubsection*{Admissible Control Policies}
An admissible \emph{storage control policy} $\pi = (\mu_0, \dots, \mu_{N-1})$ is any finite sequence of decision functions that causally map from the available information to actions, and respect constraints on both the input to and state of storage. 
We define the \emph{system state} at period $k$ as the pair $(z_k, \xi_k) \in \Rset_+ \times \Rset_+$, where we recall that $z_k$ represents the energy storage state just preceding period $k$, while $\xi_k$ denotes the intermittent supply realized during period $k$. We assume \emph{perfect state feedback} and consider control policies with \emph{full information history}. Namely, the information available to any controller at time $k$ is the vector $I_k := (x, \ z_{\leq k}, \ \xi_{\leq k})$, where $z_{\leq k} = (z_0, \dots, z_k)$ and $\xi_{\leq k} = (\xi_0, \dots, \xi_k)$. Naturally, we  allow the control policy to depend explicitly on the forward contract $x$ . A control policy $\pi$ thus defines the map
\begin{align*}
u_k = \mu_k(I_k) \quad \text{for } \ k=0,1,\dots,N-1,
\end{align*}
where $u_k \in \Rset$ is the \emph{input} to the storage system at time $k$.

 We now characterize  the space of admissible control policies, as determined by the storage type  $\theta = (b,r)$. We define the \emph{feasible state space} $\Zcal(b)$ as the set of all energy storage states respecting the energy capacity constraint. Namely, $\Zcal(b) = \{ z \in \Rset_+ \ | \  0 \leq z  \leq b \}$. Given an energy storage state  $z \in \Zcal(b)$, we define the corresponding \emph{feasible input space}  as the set of all inputs belonging to 
\[ \Ucal(z; \theta) \ = \ \{ u \in \Rset \ | \  z - u \in \Zcal(b), \ \ |u| \leq r \}, \]
which guarantees one-step state feasibility and input rate constraint satisfaction.

\begin{definition}[Admissible policies] A control policy $\pi = (\mu_0, \dots, \mu_{N-1})$ is deemed \emph{admissible} if 
\[ \mu_k(I_k) \in \Ucal(z_k; \theta)   \]
almost surely for all $I_k$ and $k=0,\dots, N-1$.
We denote by $\Pi(\theta)$ the space of all admissible control policies with full information history.
\end{definition}

\subsubsection*{Criterion}
We define the \emph{expected profit}  $J^{\pi}(x; \theta)$ derived by a supplier, with a storage of type $\theta$, as the revenue derived from a forward contract offering $x$ less the expected imbalance cost incurred under an admissible storage control policy $\pi \in \Pi(\theta)$. More precisely, we define the expected profit as
\begin{align}
\label{eq:profit}
J^{\pi}(x; \theta) \ = \   N \cdot px  \ - \ \mathbb{E}\left[ \ \sum_{k=0}^{N-1} g(x, u_k^{\pi}, \xi_k) \ \right],
\end{align}
where expectation is taken with respect to $(\alpha, \beta, \v{\xi})$ and $g(x, u_k^{\pi}, \xi_k)$ denotes the imbalance cost realized at each time  period $k$. More precisely, we have
\begin{align}
g(x, u, \xi) \ = \  \alpha \left( x- \xi - u \right)^+ \ + \ \beta \left(\xi + u  - x \right)^+.
\end{align}
Notice that the stage cost $g$ is indeed a convex function of its arguments. 
For notational concision, we suppress the dependency of $g$ on the imbalance prices $(\alpha, \beta)$. In addition, we will occasionally write the storage state and control processes as  $\{z_k^{\pi}\}$ and $\{u_k^{\pi}\}$ to emphasize their dependence on the storage control policy $\pi$.

 We wish to characterize forward contract offerings $x$ and control policies $\pi$ that together yield a maximum expected profit. This amounts to the solution of a \emph{two-stage stochastic program}, where the recourse problem constitutes a constrained stochastic control problem. Problem optimality is defined as follows.

\begin{definition}[Optimality]  \label{def:opt} An admissible pair  $(\pi^*, x^*) \in \Pi(\theta) \times \Rset$ is deemed optimal if
\[ J^{\pi^*}(x^*; \theta) \ \geq  \ J^{\pi}(x; \theta) \quad \text{for all}  \ \  (\pi, x) \in \Pi(\theta) \times \Rset. \]
We will occasionally write the optimal value function and an optimal solution pair as $J^*(\theta)$ and $(\pi^*(\theta), x^*(\theta))$, respectively, to emphasize their parametric dependency on the storage type parameter $\theta$. 

\end{definition}

The optimal forward contract may be non-unique. In order to avoid technical issues associated with such non-uniqueness,  
we will  restrict our attention to smallest contract among all optimal contracts for each $\theta$. Specifically,  define the minimal optimal contract as $x^*(\theta) = \inf \Xcal^*(\theta)$, where $\Xcal^*(\theta) = \{ x \in \Rset \ | \  J^{\pi^*(\theta)}(x; \theta)  \geq J^{\pi^*(\theta)}(y; \theta)  \quad \forall \  y \in \Rset\}$ denotes the set of all optimal contracts associated with a storage type $\theta$.

%% file: results_structure.tex
\section{Optimal Contract Properties} \label{mainresults}

We now characterize the optimal storage control policy and establish concavity of the expected profit criterion in the forward contract $x$.

\subsection{Optimal Contract Sizing for $b = 0$}
Consider first the special case of optimal contract sizing in the absence of storage, i.e. $b=0$. Naturally, in the absence of storage capacity, the set of admissable storage control policies is identically zero. And the problem of selecting a forward contract to maximize the expected profit reduces to a so-called \emph{newsvendor problem} \cite{petruzzi}. The  convexity of this optimization problem is guaranteed under our postulated assumptions. We have the following result established in \cite{bitar_trans, Dent2011}.
\begin{lemma} 
\label{lem:quantile} Consider a storage type $\theta \in \Rset^2_+$ with $b=0$.
The corresponding optimal contract  is given by the quantile
\begin{align} \label{eq:optimalcontract1}
  x^*(\theta) =  F^{-1}\left(  \gamma \right) := \inf \{ x \in \Rset \ | \ F(x) \geq \gamma\},
 \end{align}
where   $\gamma := (p+m_{\beta})/(m_{\alpha}+m_{\beta}) \in [0,1]$.   
\end{lemma}
The quantile structure of the optimal contract will prove essential to characterizing the marginal value of storage at the origin, which we present in Theorem \ref{thm:mv}.
 
 \blue{
 \begin{remark}[Supply Function Offer] \label{rem:supplyfunc}
 It is also worth noting that the optimal contract specified in \eqref{eq:optimalcontract1} 
 is a monotone nondecreasing function  in the DA market price $p$. Hence, the optimal contract can be equivalently interpreted as a \emph{supply function offer} in the DA market, which indicates the maximum amount of energy
that the WPP is willing to produce given a price $p$. 
Accordingly, all of the results presented in this paper can be shown to hold for the more general setting in which the WPP does not have explicit knowledge of the DA market price, but rather offers a supply function into the DA market, which specifies the amount it is willing to produce as a function of price. It is important to note that the validity of this interpretation is reliant upon the assumption that the
WPP behaves as a price taker in the DA market, which ensures that it has
no influence on the determination of the DA market price.
\end{remark}
}

\subsection{Optimal Contract Sizing for $b > 0$}
 In the presence of positive storage capacity, $b >0$, the selection of an optimal forward contract will naturally depend on the storage type implicitly through the choice of optimal control policy $\pi^* \in \Pi(\theta)$. We now characterize the optimal control policy in Proposition \ref{prop:optpolicy}.

\begin{proposition}[Optimal control policy] \label{prop:optpolicy} Given a storage type $\theta \in \Rset^2_+$,  the optimal control policy $\pi^* = (\mu^*, \dots, \mu^*) \in \Pi(\theta)$ is  (i) \emph{myopic} and (ii) of a  \emph{threshold-type} satisfying
\begin{align} \label{eq:optpolicy}
\mu^*(x,z, \xi) = \begin{cases} \hspace{.14in} \min\{x - \xi, \;  z, \; r \} , & \xi \leq x  \\  
-  \min\{ \xi -x, \; b- z, \; r \}, & \xi  > x  \end{cases}  
\end{align}
for all $(x,z, \xi) \in \Rset_+ \times \Zcal(b) \times \Rset_+$.
\end{proposition} 
We omit the proof of Proposition \ref{prop:optpolicy}, as optimality of the  control policy \eqref{eq:optpolicy} can be shown by direct inspection of the corresponding dynamic programming equations.

While we do not offer an explicit expression for the optimal contract size  
in this more general setting, we establish in Theorem \ref{thm:convex1} concavity of the expected profit criterion $J^{\pi^*}(x;\theta)$  in the contract size $x$, under the optimal control policy $\pi^*$ specified in \eqref{eq:optpolicy}.  See Appendix \ref{app:convex1} for a proof of Theorem \ref{thm:convex1}.

\begin{theorem}[Convexity of optimal contract sizing]  \label{thm:convex1}
Let $\pi^* \in \Pi(\theta)$ denote the optimal control policy associated with a storage of type $\theta$ and a particular forward contract $x \in \Rset_+$. Then, the expected profit $J^{\pi^*}(x;\theta)$ is a concave function in $x$ over $\Rset_+$. 
\end{theorem}

It follows from Theorem \ref{thm:convex1} that an optimal contract can be computed by solving a finite-dimensional, unconstrained convex optimization problem given by $\sup_{x \in \Rset} \left\{ J^{\pi^*}(x;\theta) \right\}.$

%% file: results_marginal_val.tex
\section{The Marginal Value of Storage}
\label{sec:marval}

\blue{As the cost required to deploy a storage facility can be large, it is of vital importance to quantify the fiscal benefit that a 
wind power producer (WPP) might derive from an initial investment in energy storage capacity.}
Theorem \ref{thm:convex1} shows that the problem of computing optimal contract offerings and the corresponding optimal expected profit is a convex program.
In the following Theorem \ref{result2}, we show that the optimal expected profit function $J^*(\theta)$  is concave and nondecreasing in the storage type  $\theta = (b,r)$.

\thm{\label{result2} The maximum expected profit $J^*(\theta)$ is concave and nondecreasing in the  storage type parameter $\theta \in \Rset_+^2$.}

We refer the reader to Appendix \ref{app:result2} for a proof of Theorem \ref{result2}.
The consequences of the Theorem \ref{result2} are twofold. \blue{First, consider the problem of optimal storage sizing $ \sup_{\theta \in \Rset_+^2}  \left\{J^*(\theta) - C(\theta)\right\}$, where $C(\theta)$ denotes the capital cost of energy storage capacity. If we assume a convex capital cost function\footnote{This is a standard assumption in the literature.  In fact, it is common to assume, more strongly, that the capital cost of energy storage capacity is linear \cite{cavallo2, cast}.}, then Theorem \ref{result2} reveals that the problem of \emph{optimal storage sizing reduces to a finite-dimensional, convex optimization problem}.} Second, the concavity and monotonicity of the maximum expected profit function $J^*(\theta)$ in the storage type $\theta$ shows that the \emph{marginal value of storage capacity  is greatest  for  initial investments in storage capacity}. In Theorem \ref{thm:mv},  we  provide a closed-form expression for the marginal value of energy storage capacity $\partial J^{*}(\theta) / \partial b$ at the origin $(b=0)$.

\subsection{$\gamma$-Quantile Level Crossings}
In Theorem \ref{thm:mv}, we make precise the intuition that  a larger variation in the intermittent supply process will manifest in a larger value of storage. In particular, we establish an explicit relationship between the   marginal value of storage capacity at the origin and a specific measure of variation of the underlying intermittent supply process. Before stating our main result, we first establish a preliminary result in Lemma \ref{lem:avg_dev}, which quantifies the expected number of times a stochastic process exceeds a fixed level over a fixed interval of time.  We have the following definition.

\begin{figure}[htb!]
\begin{center}
\includegraphics[width = 0.845 \linewidth]{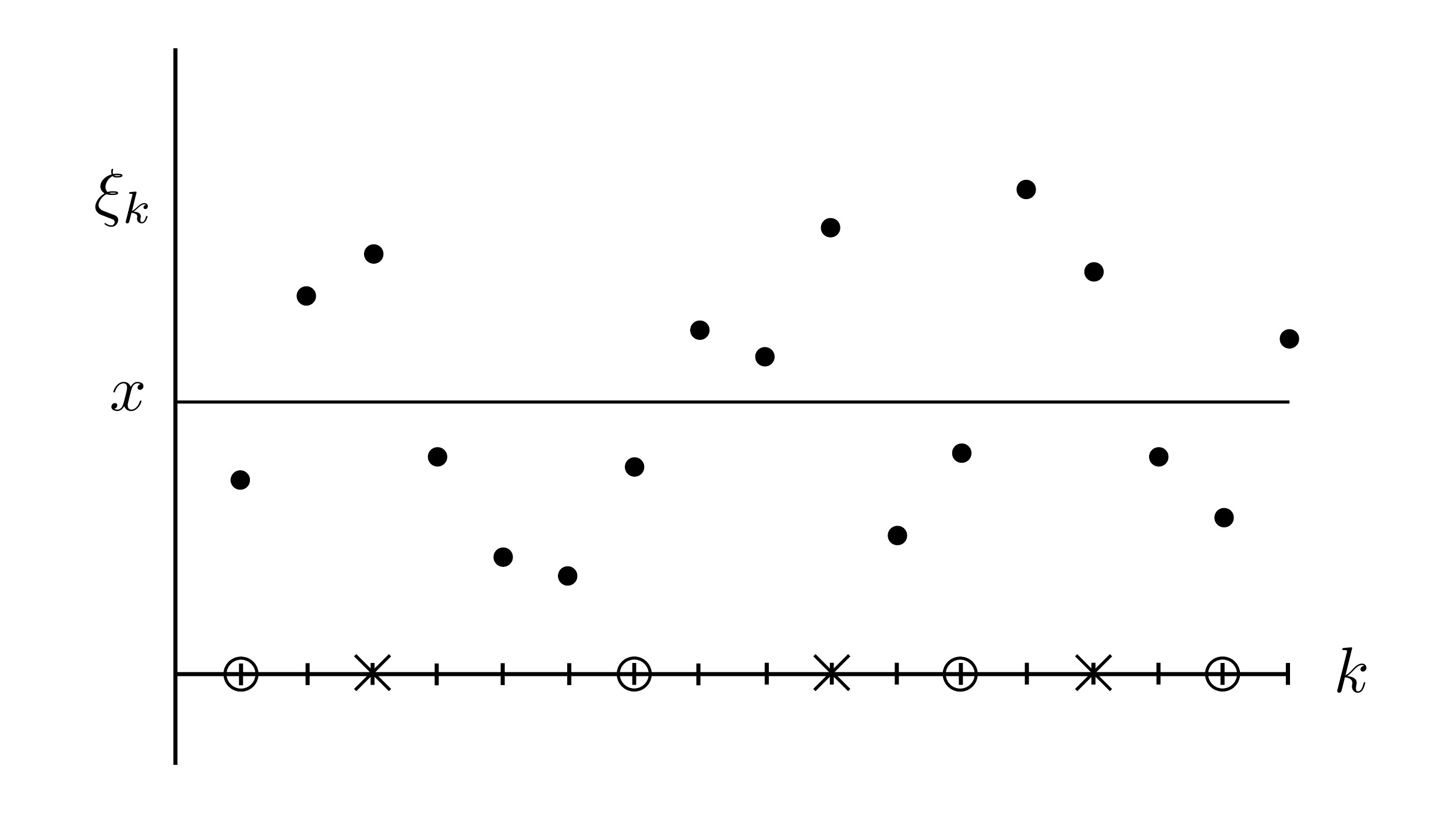}
\caption{A graphical illustration of the times at which a process $\v{\xi} = \{\xi_k\}$ exhibits strict upcrossings ($\circ$) and strict downcrossings ($\times$) of a fixed level $x$. For this example,  we have that $\Lambda_N(x, \v{\xi}) = 3$. }
\label{fig:cross}
\end{center}
\end{figure}

\begin{definition}[Strict level crossing] \label{def:strictcross} A scalar sequence $\v{a} = (a_0, \dots, a_{N-1})$ is said to have a \emph{strict downcrossing} of the level $x \in \Rset$ at time $k$ if $a_k >  x$ and $a_{k+1} < x$. Thus, a strict downcrossing of $x$ at time $k$ corresponds to the event $x \in \Dcal_k(\v{a})$, where we define $\Dcal_k(\v{a}) = \{ y \in \Rset \ | \ a_{k} > y > a_{k+1} \}$. Analogously, the sequence $\v{a}$ is said to have a \emph{strict upcrossing} of the level $x$ at time $k$ if  $x \in \Ucal_k(\v{a})$, where we define $\Ucal_k(\v{a}) = \{ y \in \Rset \ | \ a_k < y < a_{k+1} \}$.
\end{definition}
See Figure \ref{fig:cross} for a graphical illustration of Definition \ref{def:strictcross}.
We make the following technical assumption in order to restrict our attention 
 to \emph{strict} level crossings.
\begin{assumption} \label{assum:cont}
The joint distribution of the intermittent supply process $\v{\xi}$ is assumed to be absolutely continuous.
\end{assumption}
Under Assumption \ref{assum:cont},  sample paths of the process $\v{\xi}$ are, with probability one, not identically equal to $x \in \Rset$ for any $k$. More precisely,  we have that $\mathbb{P} ( \bigcup_{k} \{ \xi_k = x \} ) \leq \sum_{k} \mathbb{P} \left\{  \xi_k = x \right\} =  0$.
Henceforth, we shall refer to all strict crossings as crossings, unless otherwise unclear from the context.

\begin{definition}[Number of Strict Downcrossings]  \label{def:cross} We denote the number of strict downcrossings of $x \in \Rset$ incurred by a scalar sequence $\v{a} = (a_0, \dots, a_{N-1})$  on the interval $\{0, \dots, N-1\}$ by
\begin{align} \label{eq:cross}
\Lambda_N(x, \mathbf{a}) \ := \ \sum_{k=0}^{N-2}  \mathbf{1}_{\Dcal_k( \mathbf{a})}(x).  
\end{align}
We will omit the subscript $N$ when it is clear from the context.
\end{definition}

 We have the following Lemma characterizing the  number of times the intermittent supply process is expected to strictly exceed, or fall below the quantile $F^{-1}(\gamma)$. First, define
\begin{align*}
\mathcal{K}^+(x, \v{\xi})  & := \{ 0 \leq k  \leq N-1 \ | \ \xi_k > x \}, \\ 
\mathcal{K}^-(x, \v{\xi})  & := \{ 0 \leq k  \leq N-1 \ | \ \xi_k < x \}
\end{align*}
as the number of times which the supply \emph{strictly exceeds}, and \emph{strictly falls below} the level $x \in \Rset$, respectively.

\begin{lemma} \label{lem:avg_dev} Let $\theta =0$ and denote by $x^* \in \Rset_+$ the corresponding optimal contract.  The following properties hold:
\vspace{.0001in}
\begin{enumerate}[(i)] \setlength{\itemsep}{.07in}
\item $ \mathbb{E}|\mathcal{K}^-(x^*, \v{\xi}) | = N \gamma$
\item $ \mathbb{E}|\mathcal{K}^+(x^*, \v{\xi}) | = N (1-\gamma)$
\item $ | \mathcal{K}^-(x^*, \v{\xi})  |  \  + \ | \mathcal{K}^+(x^*, \v{\xi})  | \ = \ N$, almost surely,
\end{enumerate}
\vspace{.1in}
\noindent where $\gamma := (p + m_{\beta})/(m_{\alpha} + m_{\beta})$.
\end{lemma}

 Lemma \ref{lem:avg_dev} reveals an interesting interpretation of the price ratio $\gamma  \in [0,1]$. Namely, in the absence of energy storage capacity (i.e., $\theta = 0$), the quantile structure of the optimal contract $x^* = F^{-1}(\gamma)$ is such that the  fraction of times at which the intermittent supply is expected to fall short of the contract is precisely equal to $\gamma$.

\begin{proof}[Proof of Lemma \ref{lem:avg_dev}] We first prove part (iii). One can write the sum as 
\[
| \mathcal{K}^-(x^*, \v{\xi})  |    +  | \mathcal{K}^+(x^*, \v{\xi})  |  =   \sum_{k=0}^{N-1}  \v{1}_{(-\infty, x^*)}(\xi_k)   +  \v{1}_{(x^*, \infty)}(\xi_k) .
\]
The result follows, as Assumption \ref{assum:cont} implies that $\v{1}_{(-\infty, x^*)}(\xi_k)  \ + \ \v{1}_{(x^*, \infty)}(\xi_k) = 1$ almost surely for all $k$. We now establish part (i) through the following string of \blue{equalities}:
\begin{align*}
\mathbb{E} |\mathcal{K}^-(x^*, \v{\xi})| &   =     \mathbb{E} \left[ \sum_{k=0}^{N-1}  \ \v{1}_{(-\infty, x^*)}(\xi_k)  \right] \\
& =    \sum_{k=0}^{N-1} \mathbb{P}\{ \xi_k < x^*\} \stackrel{(a)}{=} \ N\cdot F(x^*)   \stackrel{(b)}{=} \ N \gamma.
\end{align*}
Here, equality  (a) follows from the definition of the time averaged distribution $F$, and (b) follows from the fact that $x^* = F^{-1}(\gamma)$ (cf. Lemma \ref{lem:quantile}). Part (ii) is an immediate consequence of parts (i) and (iii), thus completing the proof.
\end{proof}

\subsection{Level Crossings and the Marginal Value of Storage}

We now characterize the marginal value of energy storage capacity at the origin. Essentially, Theorem \ref{thm:mv} reveals that the marginal value of initial investment in energy storage capacity  depends on the statistical variation of supply, as measured through its expected number of strict contract downcrossings. To the best of our knowledge, Theorem \ref{thm:mv} is the first explicit characterization of the value of storage under general distributional assumptions on the intermittent supply process---requiring only that said process have an absolutely continuous joint distribution. The marginal value characterization \eqref{eq:margval} holds for general  nonstationary processes. This is 
a point of practical importance, as the behavior of  wind and solar power processes have been observed to be far from  normal or stationary.  Furthermore, it is straightforward to construct a consistent empirical estimator of the marginal value statistic \eqref{eq:margval} from time series data.

\begin{theorem}[Marginal Value at the Origin] \label{thm:mv} Let $r>0$.  The marginal value of energy storage capacity at the origin $(b=0)$ exists and is given by
\begin{align} \label{eq:margval}
\left. \frac{ \partial J^*(\theta)}{\partial b} \right|_{b=0} \hspace{-.1in} = \ (m_{\alpha} + m_{\beta})  \mathbb{E}[  \Lambda(x^*, \v{\xi})]     +   m_{\beta} \mathbb{P}\{ \xi_{N-1} > x^* \},  
\end{align}
where $x^* = F^{-1}(\gamma)$  \ and \ $\gamma := (p+m_{\beta})/(m_{\alpha} + m_{\beta})$. 
\end{theorem}

See Appendix \ref{app:margval} for a proof of Theorem \ref{thm:mv}.
 Theorem \ref{thm:mv} has an appealing interpretation. The marginal value of energy storage capacity at the origin  is proportional to the expected number of \emph{energy arbitrage opportunities} -- or, equivalently, the expected number of contract downcrossings. As an illustrative example, consider a system with a small amount of energy  
capacity, $b = \varepsilon >0$, and power capacity, $r > \varepsilon$. Each time the intermittent supply process  crosses the contract from above, one has the opportunity to inject an $\varepsilon$ amount of energy into the storage system to decrement the surplus penalty by $\beta \cdot \varepsilon$. This contract downcrossing event  is also accompanied by the additional opportunity to extract $\varepsilon$ energy from the storage device and thus decrement the shortfall penalty by $ \alpha \cdot \varepsilon$. Clearly then, the total realized  benefit for small storage capacity is roughly equal to $(\alpha + \beta)\cdot \varepsilon$   multiplied by the number of energy arbitrage opportunities. The exact derivation of the marginal value of storage at the origin is more complex, however, as one has to additionally account for the sensitivity of the optimal contract to the storage size.

\blue{
\begin{remark}[Lossy Storage Systems] \label{rem:nonid} Theorem \ref{thm:mv} can be  extended to accommodate inefficiencies in storage. Consider the following generalization of our original storage model:
\beq
\label{eq:storage_loss} z_{k+1}\; =\; \lambda z_k   \; - \; \frac{1}{\eta_{\rm out}} (u_k)^+  \; - \; \eta_{\rm in} (u_k)^- 
\eeq
for $k = 0, \dots, N-1$. Here  the scalar $\lambda \in (0,1]$ represents a leakage coefficient; and the scalars $\eta_{\rm out} \in (0,1]$ and $\eta_{\rm in} \in (0,1]$  represent the conversion efficiency of energy extraction and injection, respectively. We recover our original storage model under a choice of parameters $\lambda = \eta_{\rm in} = \eta_{\rm out} =1$. Working in this more general setting, one can easily establish a generalization of the marginal value result in Theorem \ref{thm:mv} as
\begin{align} \label{eq:mvlossy}
\left. \frac{ \partial J^*(\theta)}{\partial b} \right|_{b=0} \hspace{-.15in} =  \ ( \rho m_{\alpha} + m_{\beta}) \mathbb{E}[  \Lambda(x^*, \v{\xi})]  +  m_{\beta} \mathbb{P}\{ \xi_{N-1} > x^* \},
\end{align}
where  $x^* = F^{-1}(\gamma)$. The parameter $\rho := \lambda \eta_{\rm in} \eta_{\rm out} \in (0,1]$  can be interpreted as a \emph{discount factor} reflecting the roundtrip inefficiency associated with an energy arbitrage opportunity (downcrossing event). Naturally, the more lossy the storage system, the lower its marginal value. We omit a  formal proof of \eqref{eq:mvlossy}, as it can be established using arguments that are analogous to those used in the proof of Theorem \ref{thm:mv}. 
\end{remark}
}

 We have the following corollary to Theorem \ref{thm:mv}, which is somewhat surprising. In the event that the intermittent supply is described by an independent and identically distributed (iid) random process, the marginal value expression (\ref{eq:margval}) reveals itself to be insensitive to the choice of probability distribution, and dependent only on the market prices.

\begin{corollary} \label{cor:ber} Let $\v{\xi}$ be and iid process. Then the marginal value of energy storage capacity at the origin $(b=0)$ satisfies
\begin{align*}
\left. \frac{ \partial J^*(\theta)}{\partial b} \right|_{b=0}  \hspace{-.1in} = \ (N-1)  (m_{\alpha} + m_{\beta})(1 -\gamma) \gamma \ + \ m_{\beta}(1-\gamma),
\end{align*}
where $\gamma := (p+m_{\beta})/(m_{\alpha} + m_{\beta})$.
\end{corollary}

\begin{proof}[Proof of Corollary \ref{cor:ber}]
Using the assumption of independence across time, the expected number of strict $x^*(0)$-level downcrossings can be expressed as
\begin{align*}
\mathbb{E}[  \Lambda(x^*(0), \v{\xi})  ] \  & = \ \sum_{k=0}^{N-2} \mathbb{E}\bmat{\mathbf{1}_{\Dcal_k(\v{\xi})}(x^*(0)) } \\ 
& = \ \sum_{k=0}^{N-2} \mathbb{P}\{  \xi_k > x^*(0), \ \xi_{k+1} < x^*(0) \}\\
& =  \ \sum_{k=0}^{N-2} \big(1 - \Phi_k( x^*(0)) \big) \cdot \Phi_{k+1}(x^*(0)). 
\end{align*}
And, as the marginal distributions are time invariant, we necessarily have equivalence between the time averaged distribution and each marginal distribution, which yields
\begin{align*}
\mathbb{E}[  \Lambda(x^*(0), \v{\xi})  ]  \  = \ \sum_{k=0}^{N-2} (1 - \gamma ) \gamma  \ = \ (N-1) \cdot (1 - \gamma )  \gamma.
\end{align*}
It similarly follows that $\mathbb{P}\{ \xi_{N-1} > x^*(0)\} = 1 - \gamma$. Direct substitution of the previous two identities into Equation \eqref{eq:margval} yields the desired result.
\end{proof}

 We establish as an intermediary result in the proof of Corollary \ref{cor:ber} that the expected number of downcrossings satisfies $\mathbb{E}[  \Lambda(x^*(0), \v{\xi})  ]   =  (N-1) \cdot (1 - \gamma )  \gamma$,   under the assumption of an iid wind power process. This structural dependency on the price ratio $\gamma \in (0,1]$ admits a simple probabilistic interpretation of the interplay between the  volatility of supply and the value of storage.  Specifically, the expected number of strict downcrossings of the optimal contract level is equal to the mean of a Binomial random variable with $N-1$ trials and success probability $ (1- \gamma)\gamma$. 
\blue{That is, under the assumption of an iid wind power process,  the sequence of  contract downcrossings can be interpreted as a sequence of independent coin flips, each of which has a success probability equal to $ (1- \gamma)\gamma$. This probability is maximized for $\gamma = 1/2$.}

%% file: conclusion.tex
\section{Conclusion} \label{conclu}
In this paper we have formulated and solved the problem of optimal contract sizing for a wind power producer (WPP)  participating in a conventional two-settlement electricity market, with \emph{co-located energy storage}. Specifically, we have shown that the problem of determining optimal contract offerings for a WPP with co-located energy storage reduces to an finite-dimensional convex optimization problem. Our results have the merit of providing key analytical insight into the trade-offs between a variety of factors such as energy storage capacity and maximum expected profit. In particular, we show the marginal value of storage capacity to be largest for initial investments, and provide an analytical characterization of this marginal value---which reveals an explicit dependency of the marginal value of storage on a certain statistical measure of  variability in the underlying wind power process.

\blue{As direction for future research, it would be of value to 
expand the framework for analysis developed in this paper to allow for time-variation in the RT imbalance prices, and the possibility of statistical correlation between the wind power and price processes. We also note that the potential value that a WPP might derive from 
energy storage goes well beyond the application
of  energy arbitrage considered in this paper. For example,
certain storage technologies posses the capability of providing voltage support
or frequency regulation services---cf. \cite{castillo2014grid} for a comprehensive survey of energy storage applications. As a challenging direction for future research, it would be of value to investigate the potential economic tradeoffs that might emerge in using storage to tap these multiple value streams.}

%% file: appendix.tex
\appendix

\subsection{Proof of Theorem \ref{thm:convex1} } \label{app:convex1}

We establish concavity of $J^{\pi^*}(x;\theta)$ directly. Fix a storage type $\theta \in \Rset_+^2$ and let $x_{(1)} \in \Rset_+$ and $x_{(2)} \in \Rset_+$ be arbitrary forward contracts.\footnote{The subscripts here are not to be confused with time indices.}   Let 
\[  x_{(\lambda)} = \lambda x_{(1)} + (1- \lambda) x_{(2)} \]
denote a convex combination of said contracts, where $ \lambda \in [0,1]$.
We denote by $\pi_{(\cdot)}^* \in \Pi(\theta)$   the optimal control policy associated with the contract $x_{(\cdot)}$. And, given any admissible policy $\pi \in \Pi(\theta)$, we let $\{z_k^{\pi}\}$ and $\{u_k^{\pi}\}$ denote the random state and input  processes induced by the policy $\pi$.  

 We  establish the desired result  by showing:
\[ J^{\pi_{(\lambda)}^*}(x_{(\lambda)};\theta)   \ \geq \    \lambda \ J^{\pi_{(1)}^*}(x_{(1)};\theta) \ + \ (1 - \lambda) \ J^{\pi_{(2)}^*}(x_{(2)};\theta). \]
Consider the forward contract $x_{(\lambda)}$.
And consider a policy  $\pi_{(\lambda)}$ inducing the input process
$
u_{k}^{\pi_{(\lambda)}}  \ = \  \lambda \ u_{k}^{\pi_{(1)}^*} \ + \ (1-\lambda) \ u_{k}^{\pi_{(2)}^*},
$
 where  the associated state process is recursively determined by $z_{k+1}^{\pi_{(\lambda)}} = z_{k}^{\pi_{(\lambda)}} - u_{k}^{\pi_{(\lambda)}}$  for $z_{0}^{\pi_{(\lambda)}} =0$. It is not difficult to see that $z_{k}^{\pi_{(\lambda)}}  =  \lambda \ z_{k}^{\pi_{(1)}^*} \ + \ (1-\lambda) \ z_{k}^{\pi_{(2)}^*} $.  Admissibility of $\pi_{(\lambda)}$  is therefore immediate, as the underlying constraints on both the state and input define convex sets.  It follows that 
 \[ J^{\pi_{(\lambda)}^*}(x_{(\lambda)};\theta)   \ \geq \ J^{\pi_{(\lambda)}}(x_{(\lambda)};\theta), \]
 by optimality of policy $\pi_{(\lambda)}^* \in \Pi(\theta)$ for the contract $x_{(\lambda)}$. Expanding the expression to the right of the inequality further, we have that
 
\begin{align}
& J^{\pi_{(\lambda)}}(x_{(\lambda)};\theta)  \notag\\
& = N \cdot px_{(\lambda)}  \ - \ \mathbb{E}\left[ \ \sum_{k=0}^{N-1} g\left(x_{(\lambda)}, u_k^{\pi_{(\lambda)}}, \xi_k \right) \ \right]  \notag\\
& \geq   N \cdot px_{(\lambda)}  \ - \ \E\left[  \ \  \sum_{k=0}^{N-1}  \lambda\ g\left(x_{(1)}, u_k^{\pi_{(1)}^*}, \xi_k \right) \right.  \notag \\  
& \hspace{.93in} + \  \ (1-\lambda) \ g\left(x_{(2)}, u_k^{\pi_{(2)}^*}, \xi_k \right)   \ \Bigg] \label{eq:inter.1} \\
& = \lambda \ J^{\pi_{(1)}^*}(x_{(1)};\theta)   \ \ + \ \  (1-\lambda) \ J^{\pi_{(2)}^*}(x_{(2)};\theta) \label{eq:inter.2},
\end{align}
where the inequality \eqref{eq:inter.1} follows from convexity of $g$ and the decomposition $\xi_k = \lambda \xi_k + (1-\lambda) \xi_k$.  The final equality \eqref{eq:inter.2} follows from the optimality of the policies $\pi_{(1)}^* \in \Pi(\theta)$ and $\pi_{(2)}^* \in \Pi(\theta)$ for the contracts $x_{(1)}$ and $x_{(2)}$ respectively. Thus, $J^{\pi^*}(x;\theta)$ is a concave function in $x$ over $\Rset_+$.

\subsection{Proof of Theorem \ref{result2} } \label{app:result2}

Monotonicity is straightforward. Fix a storage type $\theta \in \Rset_+^2$. Let $\varepsilon \in \Rset_+^2$.  Clearly, $\Pi(\theta + \varepsilon) \supseteq \Pi(\theta)$ and hence $J^*(\theta + \varepsilon) \geq J^*(\theta)$. 
The proof of concavity of $J^*(\theta)$ in $\theta$ over $\Rset_+^2$ is analogous to the proof of Theorem \ref{thm:convex1}.

\subsection{Proof of Theorem \ref{thm:mv}} \label{app:margval}
Fix $r >0$, and write $x^*(b) = x^*(\theta)$ and $\pi^*(b) = \pi^*(\theta)$ to isolate their dependence on the energy capacity parameter $b$, as we have fixed $r$. \blue{It will be  convenient to decompose the optimal expected profit associated with a storage type $\theta$ as 
\begin{align*}
J^*(\theta) =   pN  x^*(b)  -  \mathbb{E}\left[ Q^{\pi^*(b)}(x^*(b), \v{\xi})  \right],
\end{align*}
where
\begin{align*}
Q^{\pi^*(b)}(x^*(b), \v{\xi})  = \sum_{k=0}^{N-1} g(x^*(b), u_k^{\pi^*(b)}, \xi_k)
\end{align*}
denotes the  imbalance cost realized under $(x^*(b), \pi^*(b), \v{\xi})$.}

We begin the proof by expressing the (right) partial derivative of $J^*(\theta)$ with respect to $b$ at the origin as
\begin{align} \notag
& \left. \frac{ \partial J^*(\theta)}{\partial b} \right|_{b=0} = \  pN \left. \frac{ \partial x^*(b)}{\partial b} \right|_{b=0}   \\
 &\hspace{.2in} - \ \lim_{\varepsilon \downarrow 0} \mathbb{E}\left[\frac{ Q^{\pi^*(\varepsilon)}(x^*(\varepsilon), \v{\xi})  -   Q^{\pi^*(0)}(x^*(0), \v{\xi}) }{\varepsilon}  \right].  \label{eq:lim_deriv}
\end{align}
We proceed in establishing existence of the limit through its explicit characterization. First, define a sequence of functions $\{f_n\}$, mapping sample paths $\v{\xi} \in [0,1]^N$ into $\Rset$, as 
\begin{align*}
f_n(\v{\xi}) \ = \  \frac{ Q^{\pi^*(\varepsilon_n)}(x^*(\varepsilon_n), \v{\xi})  -  Q^{\pi^*(0)}(x^*(0), \v{\xi}) }{\varepsilon_n} , \quad  \ n \in \Nset
\end{align*}
where $\{\varepsilon_n\}$ is a sequence of non-negative real numbers converging  monotonically to zero. We will prove Theorem \ref{thm:mv} through application of the  Bounded Convergence Theorem. We first have the following result establishing almost sure convergence  and uniform boundedness of the sequence $\{f_n(\v{\xi})\}$. 
See Appendix \ref{app:as} for its proof.

\begin{proposition} \label{prop:as} Let $(\Rset, \Bcal(\Rset), \mathbb{P})$ denote the complete probability space according to which the random variables  $\v{\xi} = (\xi_0, \xi_1, \dots,\xi_{N-1})$ are defined, where $\Bcal(\Rset)$ denotes the Borel $\sigma$-algebra on $\Rset$. \blue{Define the sequence of functions $\{f_n\}$ according to
\begin{align*}
f_n(\v{\xi}) \ = \  \frac{ Q^{\pi^*(\varepsilon_n)}(x^*(\varepsilon_n), \v{\xi})  -  Q^{\pi^*(0)}(x^*(0), \v{\xi}) }{\varepsilon_n} , \quad  \ n \in \Nset
\end{align*}
where $\{\varepsilon_n\}$ is a sequence of non-negative real numbers converging  monotonically to zero.}
It follows that:

\

\noindent (i) \ $\{f_n(\v{\xi})\}$ is  a sequence of real-valued random variables converging \emph{almost surely} to the real-valued random variable $f(\v{\xi})$ defined by
\begin{align*} 
 & f(\v{\xi})   =   \left. \frac{ \partial x^*(b)}{\partial b} \right|_{b=0} \left( \alpha \cdot | \Kcal^-(x^*(0), \v{\xi})|    -   \beta \cdot | \Kcal^+(x^*(0),\v{\xi})| \right)  \\   
 & - \   (\alpha + \beta) \cdot \Lambda(x^*(0), \v{\xi}) \ - \ \beta \cdot \v{1}_{(x^*(0), \infty)}(\xi_{N-1}).
\end{align*}
The random variables $f(\v{\xi})$ and $\{f_n(\v{\xi})\}$ are defined on the common probability space $(\Rset, \Bcal(\Rset), \mathbb{P})$.

\

\noindent (ii) There exists a constant $M < \infty$, such that  $| f_n(\v{\xi})| < M$ almost surely. 

\end{proposition}

 It follows from Prop. \ref{prop:as} and the Bounded Convergence Theorem that $\lim_{n \rightarrow \infty} \mathbb{E}[f_n(\v{\xi})] = \mathbb{E}[f(\v{\xi})]$. Note that it suffices for uniform boundedness to hold almost surely, as the underlying probability space is complete by assumption. Finally, it follows from Lemma \ref{lem:avg_dev} that 
\begin{align*}
&\mathbb{E}\bmat{ \alpha \cdot | \Kcal^-(x^*(0), \v{\xi})|  \  -  \ \beta \cdot | \Kcal^+(x^*(0),\v{\xi})| } = pN. 
\end{align*}
The desired result follows.

\subsection{Proof of Proposition \ref{prop:as}}  \label{app:as}

Throughout the proof, we restrict our attention to only those sample paths $\v{\xi}$, which exhibit strict crossings of the contract $x^*(0)$. Accordingly, define the set of sample paths $\v{\xi}$ that are nowhere equal to $x^*(0)$ as 
$\Scal = \{ \v{\xi} \in [0,1]^N \ | \ \xi_k \neq x^*(0) \ \ \forall \ k \}$.
 It follows from Assumption \ref{assum:cont} that $\mathbb{P}\{  \xi \in \Scal\} = 1$.

 \emph{Proof of Part (i).} \  \ We first show that the sequence of functions $\{f_n\}$ converges pointwise to $f$ on $\Scal$. Fix $\v{\xi} \in \Scal$. We begin by controlling the behavior of the sequence $\{f_n(\v{\xi})\}$ for $n$ large enough. 
Denote by 
\begin{align*}
\Ical^-_n(\v{\xi}) & = \{ 0 \leq k \leq N-2 \ | \ \mathbf{1}_{\Dcal_k(\v{\xi})}(x^*(\varepsilon_n)) = 1 \}, \ \ \text{and} \\
\Ical^+_n(\v{\xi}) & = \{ 0 \leq k \leq N-2 \ | \ \mathbf{1}_{\Ucal_k(\v{\xi})}(x^*(\varepsilon_n))  = 1 \}
\end{align*}
the  collection of time indices at which the sample path $\v{\xi}$ exhibits strict downcrossings and upcrossings of the level $x^*(\varepsilon_n)$, respectively.

It follows from the right continuity of the optimal contract $x^*(b)$ at $b=0$ that there exists an integer $N_1(\v{\xi}) \in \Nset$  
\blue{such  that  
$\Ical^{+}_n(\v{\xi}) = \Ical^{+}_{N_1(\v{\xi})} (\v{\xi})$ and $\Ical^{-}_n(\v{\xi}) = \Ical^{-}_{N_1(\v{\xi})} (\v{\xi})$
for all $n \geq N_1(\v{\xi})$.} 
%
It will also be useful to define the integer
\[ N_2(\v{\xi}) \ = \ \min \{  n \geq N_1(\v{\xi})  \ | \  \varepsilon_n \leq \min\{r, \ | \xi_k - x^*(\varepsilon_n) | \;\} \ \forall \ k \}. \]
Essentially,  $n \geq N_2(\v{\xi})$ ensures that  any  downcrossing (upcrossing) of the level $x^*(\varepsilon_n)$ will result in a full discharge (charge) of the energy storage in the amount of $\varepsilon_n$ under the optimal control policy $\pi^*(\varepsilon_n)$. 
We now derive a closed form expression for the optimal imbalance cost $Q^{\pi^*(\varepsilon_n)}(x^*(\varepsilon_n), \v{\xi})$ for $n \geq N_2(\v{\xi})$. 

Let $n \geq N_2(\v{\xi})$. 
It is not difficult to see that, under the optimal control policy $\pi^*(\varepsilon_n)$ (cf. Prop. \ref{prop:optpolicy}), the storage system is fully charged (discharged) only at times immediately following a strict $\v{\xi}$ upcrossing (downcrossing)  of $x^*(\varepsilon_n)$.
More precisely, the sequence of optimal control inputs can be explicitly expressed as
\begin{align} \label{eq:opt_control}
&u_k^{\pi^*(\varepsilon_n)}  \\
\notag & = \left\{ \begin{array}{ll}  \varepsilon_n \cdot \mathbf{1}_{  (- \infty, \xi_k) }(x^*(\varepsilon_n)), &   k =0 \\
\varepsilon_n \cdot \mathbf{1}_{   \Ucal_{k-1}(\v{\xi})}(x^*(\varepsilon_n))  \ - \ \varepsilon_n \cdot \mathbf{1}_{   \Dcal_{k-1}(\v{\xi})}(x^*(\varepsilon_n))   , & k >0 \end{array} \right.
\end{align}
for $k =0, \dots, N-1$.
Substituting Equation (\ref{eq:opt_control}) into our nominal expression for the imbalance cost
\[ Q^{\pi^*(\varepsilon_n)}(x^*(\varepsilon_n), \v{\xi})  = \sum_{k=0}^{N-1} g(x^*(\varepsilon_n), u_k^{\pi^*(\varepsilon_n)}, \xi_k), \] 
 we have
\begin{align} \nonumber  \label{eq:inc_cost} 
& f_n(\v{\xi}) = - \  (\alpha + \beta)  \Lambda(x^*(\varepsilon_n) , \v{\xi})  
   - \ \beta  \v{1}_{(x^*(\varepsilon_n), \infty)}(\xi_{N-1}) \\
\notag &  + \bigg( \alpha  | \mathcal{K}^-(\v{\xi}, x^*(\varepsilon_n) ) |      -  \beta  | \mathcal{K}^+(\v{\xi}, x^*(\varepsilon_n) ) | \bigg)  \left( \frac{x^*(\varepsilon_n)  - x^*(0)}{\varepsilon_n} \right) 
\end{align}
The pointwise convergence $\{f_n\}$ to $f$ on $\Scal$ follows from the fact that 
$$ \mathcal{K}^{\pm}(\v{\xi}, x^*(\varepsilon_n) ) = \mathcal{K}^{\pm}(\v{\xi}, x^*(0) ),$$
for  $\v{\xi} \in \Scal$ and $n \geq N_2(\v{\xi})$.  

We now show that the convergence is \emph{almost sure}.
First notice that each function $f_n: \Rset^N \rightarrow \Rset$---being a finite linear combination of indicator functions defined on Borel measurable sets---is Borel measureable.
Since the composition of  measurable functions is measurable, it follows that the composition $f_n(\v{\xi})$ is a random variable on $(\Rset, \Bcal(\Rset), \mathbb{P})$. The same is true for the pointwise limit function $f(\v{\xi})$. Almost sure convergence is immediate as
\begin{align*}
&\mathbb{P}\left\{  \lim_{n \rightarrow \infty} f_n(\v{\xi}) = f( \v{\xi}) \right\} \\
& \qquad \geq  \mathbb{P}\left\{  \left. \lim_{n \rightarrow \infty} f_n(\v{\xi}) = f( \v{\xi})  \right|    \v{\xi} \in \Scal   \right\}\mathbb{P}\{ \v{\xi} \in \Scal \} = 1. 
\end{align*}
This completes the proof of Part (i).

\emph{Proof of Part (ii).} \ \  We first show uniform boundedness of each function $f_n : \Rset^N \rightarrow \Rset$ on $\Scal$. By assumption, we have restricted the image of each random variable $\xi_k$ to $[0,1]$ for all $k$. As an immediate consequence, we have that $ 0 \leq x^*(b) \leq 1$ for all $b\geq 0$. It follows from this fact, and the right differentiability of $x^*(b)$ at $b=0$, that
\[  B := \sup\left\{ \left. \frac{|x^*(\varepsilon_n) - x^*(0)|}{\varepsilon_n} \ \right| \   n \in \Nset \right\} \ < \ \infty. \]
Combining this upper bound with the observation that $| \mathcal{K}^-(\v{\xi}, x^*(\varepsilon_n))|$,  $| \mathcal{K}^+(\v{\xi}, x^*(\varepsilon_n))|$, $ \Lambda(\v{\xi}, x^*(\varepsilon_n)) \leq N$ for all $n \in \Nset$ and  $\v{\xi} \in \Scal$, we have that 
\[ | f_n(\v{\xi}) | \ \leq \  NB \cdot (\alpha + \beta) +  N \cdot (\alpha + \beta) + \beta \]
for all $n \in \Nset$ and  $\v{\xi} \in \Scal$.
Uniform boundedness holds almost surely as $\mathbb{P}\{ \v{\xi} \in \Scal \} = 1$. This completes the proof of Part (ii).

%% file: SCL_Main.bbl
\begin{thebibliography}{10}
\providecommand{\url}[1]{#1}
\csname url@samestyle\endcsname
\providecommand{\newblock}{\relax}
\providecommand{\bibinfo}[2]{#2}
\providecommand{\BIBentrySTDinterwordspacing}{\spaceskip=0pt\relax}
\providecommand{\BIBentryALTinterwordstretchfactor}{4}
\providecommand{\BIBentryALTinterwordspacing}{\spaceskip=\fontdimen2\font plus
\BIBentryALTinterwordstretchfactor\fontdimen3\font minus
  \fontdimen4\font\relax}
\providecommand{\BIBforeignlanguage}[2]{{%
\expandafter\ifx\csname l@#1\endcsname\relax
\typeout{** WARNING: IEEEtran.bst: No hyphenation pattern has been}%
\typeout{** loaded for the language `#1'. Using the pattern for}%
\typeout{** the default language instead.}%
\else
\language=\csname l@#1\endcsname
\fi
#2}}
\providecommand{\BIBdecl}{\relax}
\BIBdecl

\bibitem{BitarACC2011}
E.~Bitar, R.~Rajagopal, P.~Khargonekar, and K.~Poolla, ``The role of co-located
  storage for wind power producers in conventional electricity markets,'' in
  \emph{American Control Conference (ACC), 2011}.\hskip 1em plus 0.5em minus
  0.4em\relax IEEE, 2011, pp. 3886--3891.

\bibitem{GWEC2010}
``Global wind 2009 report,'' \emph{Global Wind Energy Council, Brussels,
  Belgium}, 2010.

\bibitem{EWITS}
{EnerNex Corp.}, ``Eastern wind integration and transmission study,''
  \emph{Report NREL/SR-550-47078}, January 2010.

\bibitem{IVGTF2009}
{Integration of Variable Generation Task Force}, ``Accommodating high levels of
  variable generation,'' \emph{North America Reliability Corporation}, 2009.

\bibitem{WWTS}
{GE Energy}, ``Western wind and solar integration study,'' \emph{National
  Renewable Energy Laboratory, Report NREL/SR-550-47434}, May 2010.

\bibitem{castillo2014grid}
A.~Castillo and D.~F. Gayme, ``Grid-scale energy storage applications in
  renewable energy integration: A survey,'' \emph{Energy Conversion and
  Management}, vol.~87, pp. 885--894, 2014.

\bibitem{cavallo2}
A.~Cavallo, ``Controllable and affordable utility-scale electricity from
  intermittent wind resources and compressed air energy storage ({CAES}),''
  \emph{Energy}, vol.~32, no.~2, pp. 120--127, 2007.

\bibitem{decesaro}
J.~DeCesaro, K.~Porter, and M.~Milligan, ``Wind energy and power system
  operations: A review of wind integration studies to date,'' \emph{The
  Electricity Journal}, vol.~22, no.~10, pp. 34--43, 2009.

\bibitem{denholmnrel}
P.~Denholm, E.~Ela, B.~Kirby, and M.~Milligan, ``The role of energy storage
  with renewable electricity generation,'' \emph{Technical Report
  NREL/TP-6A2-47187}, 2010.

\bibitem{divya}
K.~Divya and J.~{\O}stergaard, ``Battery energy storage technology for power
  systems: An overview,'' \emph{Electric Power Systems Research}, vol.~79,
  no.~4, pp. 511--520, 2009.

\bibitem{korpaas03}
M.~Korpaas, A.~T. Holen, and R.~Hildrum, ``Operation and sizing of energy
  storage for wind power plants in a market system,'' \emph{International
  Journal of Electrical Power \& Energy Systems}, vol.~25, no.~8, pp. 599--606,
  2003.

\bibitem{cast}
E.~D. Castronuovo and J.~A.~P. Lopes, ``Optimal operation and hydro storage
  sizing of a wind--hydro power plant,'' \emph{International Journal of
  Electrical Power \& Energy Systems}, vol.~26, no.~10, pp. 771--778, 2004.

\bibitem{angarita1}
J.~L. Angarita, J.~Usaola, and J.~Mart{\'\i}nez-Crespo, ``Combined hydro-wind
  generation bids in a pool-based electricity market,'' \emph{Electric Power
  Systems Research}, vol.~79, no.~7, pp. 1038--1046, 2009.

\bibitem{Bathurst2002}
G.~N. Bathurst, J.~Weatherill, and G.~Strbac, ``Trading wind generation in
  short term energy markets,'' \emph{IEEE Transactions on Power Systems},
  vol.~17, no.~3, pp. 782--789, 2002.

\bibitem{bitar_trans}
E.~Y. Bitar, R.~Rajagopal, P.~P. Khargonekar, K.~Poolla, and P.~Varaiya,
  ``Bringing wind energy to market,'' \emph{IEEE Transactions on Power
  Systems}, vol.~27, no.~3, pp. 1225--1235, 2012.

\bibitem{Dent2011}
C.~J. Dent, J.~W. Bialek, and B.~F. Hobbs, ``Opportunity cost bidding by wind
  generators in forward markets: Analytical results,'' \emph{IEEE Transactions
  on Power Systems}, vol.~26, no.~3, pp. 1600--1608, 2011.

\bibitem{Matevosyan2006}
J.~Matevosyan and L.~Soder, ``Minimization of imbalance cost trading wind power
  on the short-term power market,'' \emph{IEEE Transactions on Power Systems},
  vol.~21, no.~3, pp. 1396--1404, 2006.

\bibitem{morales10}
J.~M. Morales, A.~J. Conejo, and J.~P{\'e}rez-Ruiz, ``Short-term trading for a
  wind power producer,'' \emph{IEEE Transactions on Power Systems}, vol.~25,
  no.~1, pp. 554--564, 2010.

\bibitem{Pag2014}
F.~Paganini, P.~Belzarena, and P.~Monz{\'o}n, ``Decision making in forward
  power markets with supply and demand uncertainty,'' in \emph{Information
  Sciences and Systems (CISS), 2014 48th Annual Conference on}.\hskip 1em plus
  0.5em minus 0.4em\relax IEEE, 2014, pp. 1--6.

\bibitem{Pinson2007}
P.~Pinson, C.~Chevallier, and G.~N. Kariniotakis, ``Trading wind generation
  from short-term probabilistic forecasts of wind power,'' \emph{IEEE
  Transactions on Power Systems}, vol.~22, no.~3, pp. 1148--1156, 2007.

\bibitem{katz}
D.~L.~V. Katz and E.~R. Lady, \emph{Compressed air storage for electric power
  generation}.\hskip 1em plus 0.5em minus 0.4em\relax Ulrich's Book Inc, 1976.

\bibitem{lund2}
H.~Lund and G.~Salgi, ``The role of compressed air energy storage ({CAES}) in
  future sustainable energy systems,'' \emph{Energy Conversion and Management},
  vol.~50, no.~5, pp. 1172--1179, 2009.

\bibitem{succar}
S.~Succar, R.~H. Williams \emph{et~al.}, ``Compressed air energy storage:
  theory, resources, and applications for wind power,'' \emph{Energy Systems
  Analysis Group, Princeton University}, 2008.

\bibitem{oshima}
T.~Oshima, M.~Kajita, and A.~Okuno, ``Development of sodium-sulfur batteries,''
  \emph{International Journal of Applied Ceramic Technology}, vol.~1, no.~3,
  pp. 269--276, 2004.

\bibitem{bose2012optimal}
S.~Bose, D.~F. Gayme, U.~Topcu, and K.~M. Chandy, ``Optimal placement of energy
  storage in the grid,'' in \emph{Decision and Control (CDC), 2012 IEEE 51st
  Annual Conference on}.\hskip 1em plus 0.5em minus 0.4em\relax IEEE, 2012, pp.
  5605--5612.

\bibitem{DenholmSioshansi2009}
P.~Denholm and R.~Sioshansi, ``The value of compressed air energy storage with
  wind in transmission-constrained electric power systems,'' \emph{Energy
  Policy}, vol.~37, no.~8, pp. 3149--3158, 2009.

\bibitem{thrampoulidis2016optimal}
C.~Thrampoulidis, S.~Bose, and B.~Hassibi, ``Optimal placement of distributed
  energy storage in power networks,'' \emph{IEEE Transactions on Automatic
  Control}, vol.~61, no.~2, pp. 416--429, 2016.

\bibitem{cavallo1}
A.~J. Cavallo, ``High-capacity factor wind energy systems,'' \emph{Journal of
  Solar Energy Engineering}, vol. 117, pp. 137--137, 1995.

\bibitem{greenblatt1}
J.~B. Greenblatt, S.~Succar, D.~C. Denkenberger, R.~H. Williams, and R.~H.
  Socolow, ``Baseload wind energy: modeling the competition between gas
  turbines and compressed air energy storage for supplemental generation,''
  \emph{Energy Policy}, vol.~35, no.~3, pp. 1474--1492, 2007.

\bibitem{de2016value}
F.~J. De~Sisternes, J.~D. Jenkins, and A.~Botterud, ``The value of energy
  storage in decarbonizing the electricity sector,'' \emph{Applied Energy},
  vol. 175, pp. 368--379, 2016.

\bibitem{del2016synergy}
P.~C. Del~Granado, Z.~Pang, and S.~W. Wallace, ``Synergy of smart grids and
  hybrid distributed generation on the value of energy storage,'' \emph{Applied
  Energy}, vol. 170, pp. 476--488, 2016.

\bibitem{lamont2013assessing}
A.~D. Lamont, ``Assessing the economic value and optimal structure of
  large-scale electricity storage,'' \emph{IEEE Transactions on Power Systems},
  vol.~28, no.~2, pp. 911--921, 2013.

\bibitem{go2016assessing}
R.~S. Go, F.~D. Munoz, and J.-P. Watson, ``Assessing the economic value of
  co-optimized grid-scale energy storage investments in supporting high
  renewable portfolio standards,'' \emph{Applied energy}, vol. 183, pp.
  902--913, 2016.

\bibitem{Qin2015}
J.~Qin, Y.~Chow, J.~Yang, and R.~Rajagopal, ``Online modified greedy algorithm
  for storage control under uncertainty,'' \emph{IEEE Transactions on Power
  Systems}, vol.~31, no.~3, pp. 1729--1743, 2016.

\bibitem{chowdhury2016benefits}
M.~Chowdhury, M.~Rao, Y.~Zhao, T.~Javidi, and A.~Goldsmith, ``Benefits of
  storage control for wind power producers in power markets,'' \emph{IEEE
  Transactions on Sustainable Energy}, vol.~7, no.~4, pp. 1492--1505, 2016.

\bibitem{KimPowell2011}
J.~H. Kim and W.~B. Powell, ``Optimal energy commitments with storage and
  intermittent supply,'' \emph{Operations research}, vol.~59, no.~6, pp.
  1347--1360, 2011.

\bibitem{harsha2015}
P.~Harsha and M.~Dahleh, ``Optimal management and sizing of energy storage
  under dynamic pricing for the efficient integration of renewable energy,''
  \emph{IEEE Transactions on Power Systems}, vol.~30, no.~3, pp. 1164--1181,
  2015.

\bibitem{parandehgheibi2015value}
A.~ParandehGheibi, M.~Roozbehani, M.~A. Dahleh, and A.~Ozdaglar, ``The value of
  storage in securing reliability and mitigating risk in energy systems,''
  \emph{Energy Systems}, vol.~6, no.~1, pp. 129--152, 2015.

\bibitem{rao2015value}
M.~Rao, M.~Chowdhury, Y.~Zhao, T.~Javidi, and A.~Goldsmith, ``Value of storage
  for wind power producers in forward power markets,'' in \emph{American
  Control Conference (ACC), 2015}.\hskip 1em plus 0.5em minus 0.4em\relax IEEE,
  2015, pp. 5686--5691.

\bibitem{van2013optimal}
P.~M. van~de Ven, N.~Hegde, L.~Massouli{\'e}, and T.~Salonidis, ``Optimal
  control of end-user energy storage,'' \emph{IEEE Transactions on Smart Grid},
  vol.~4, no.~2, pp. 789--797, 2013.

\bibitem{Baeyens2013}
E.~Baeyens, E.~Y. Bitar, P.~P. Khargonekar, and K.~Poolla, ``Coalitional
  aggregation of wind power,'' \emph{IEEE Transactions on Power Systems},
  vol.~28, no.~4, pp. 3774--3784, 2013.

\bibitem{petruzzi}
N.~C. Petruzzi and M.~Dada, ``Pricing and the newsvendor problem: A review with
  extensions,'' \emph{Operations research}, vol.~47, no.~2, pp. 183--194, 1999.

\end{thebibliography}
